\documentclass[11pt]{amsart}

\usepackage{a4wide}
\usepackage{amssymb}
\usepackage{mathrsfs}
\usepackage{enumerate}
\usepackage{esint}
\usepackage{titletoc}
\usepackage[colorlinks=true, urlcolor=blue, linkcolor=blue, citecolor=green]{hyperref}
\usepackage[nameinlink]{cleveref}
\usepackage{mathtools}
\theoremstyle{plain}
\newtheorem{theorem}{Theorem}[section]
\newtheorem{lemma}[theorem]{Lemma}

\theoremstyle{definition}
\newtheorem{definition}[theorem]{Definition}

\numberwithin{equation}{section}
\DeclareMathOperator*{\osc}{osc}

\makeatletter
\@namedef{subjclassname@2020}{\textup{2020} Mathematics Subject Classification}
\makeatother

\title[Regularity for functions in solution classes]{$C^{1, \alpha}$-regularity for functions in solution classes and its application to parabolic normalized $p$-Laplace equations}

\author{Se-Chan Lee}
\address{Research Institute of Mathematics,
	Seoul National University, Seoul 08826, Republic of Korea.}
\email{dltpcks1@snu.ac.kr}

\author{Hyungsung Yun}
\address{Department of Mathematical Sciences,
	Seoul National University, Seoul 08826, Republic of Korea.}
\email{euler@snu.ac.kr}

\subjclass[2020]{35B65; 35D40; 35K92}
\keywords{Solution class; global regularity; parabolic normalized $p$-Laplace equations}
\thanks{Se-Chan Lee is supported by Basic Science Research Program through the National Research Foundation of Korea (NRF) funded by the Ministry of Education (2022R1A6A3A01086546).}

\begin{document}

\begin{abstract}
	We establish the global $C^{1, \alpha}$-regularity for functions in solution classes, whenever ellipticity constants are sufficiently close. As an application, we derive the global regularity result concerning the parabolic normalized $p$-Laplace equations, provided that $p$ is close to 2. Our analysis relies on the compactness argument with the iteration procedure.
\end{abstract}

\maketitle

%%%%%%%%%%%%%%%%%%%%%%%%%%%%%%%%%%%%%%%

\section{Introduction}
In this paper, we are concerned with regularity results for functions $u \in C(Q_1)$ satisfying the following inequalities in the viscosity sense:
\begin{equation}\label{eq:s*_class}
	\left\{\begin{aligned}
		u_t - \mathcal{M}^{-}_{\lambda, \Lambda}(D^2u) & \ge- \|f\|_{L^{\infty}(Q_1)} \\
		u_t -\mathcal{M}^{+}_{\lambda, \Lambda}(D^2u)  & \leq \|f\|_{L^{\infty}(Q_1)} 
	\end{aligned}\right.
	\quad \text{in $Q_1$},
\end{equation}
where $f \in C(Q_1) \cap L^{\infty}(Q_1)$ and $\mathcal{M}_{\lambda, \Lambda}^{\pm}$ are the Pucci's extremal operators. Here we define \textit{Pucci's extremal operators} as follows: for ellipticity constants $0<\lambda \leq \Lambda$ and $\mathcal{S}^n \coloneqq \{ M : \text{$M$ is a $n\times n$ real symmetric matrix\}}$,
\begin{align*}
	\mathcal{M}_{\lambda, \Lambda}^{+}(M)  \coloneqq \sup_{\lambda I \leq A \leq \Lambda I} \text{tr} (AM)  \quad \text{and} \quad
	\mathcal{M}^{-}_{\lambda, \Lambda}(M)  \coloneqq \inf_{\lambda I \leq A \leq \Lambda I} \text{tr} (AM).
\end{align*}
To be precise, we define the \textit{extended solution class} $S^{\ast}(\lambda, \Lambda, f)$ by the space of functions $u \in C(Q_1)$ satisfying \eqref{eq:s*_class} in the viscosity sense. In a similar way, we define the \textit{solution class} $S(\lambda, \Lambda, f)$ by the space of functions $u \in C(Q_1)$ satisfying 
\begin{equation} \label{eq:s_class}
	u_t -\mathcal{M}^{+}_{\lambda, \Lambda}(D^2u) \leq f \leq u_t - \mathcal{M}^{-}_{\lambda, \Lambda}(D^2u)   \quad \text{in $Q_1$}
\end{equation}
in the viscosity sense; see \Cref{sec:preliminaries} for details. We remark that if $u \in C(Q_1)$ is a viscosity solution of 
\begin{equation} \label{eq:model}
	u_t - F(D^2u, x,t)=f  \quad \text{in } Q_1
\end{equation}
for some $(\lambda, \Lambda)$-elliptic fully nonlinear operator $F$, then we have
$u \in S(\lambda, \Lambda, f(x,t)+F(0, x,t))$. On the other hand, there exists a function $u \in C(Q_1)$ which belongs to the solution class $S(\lambda, \Lambda, 0)$, but cannot be a viscosity solution of 
\begin{align*}
	u_t-F(D^2u)=0 \quad \text{in $Q_1$}
\end{align*}
for any $(\lambda, \Lambda)$-elliptic operators $F$. See \cite[Section 1.7.2]{Pim22} for the precise construction of such an example. 

Before we state our main theorem, we summarize several regularity results concerning the functions in the solution classes for both elliptic and parabolic settings:
\begin{enumerate}[(i)]	
	\item (Interior $W^{2, \delta}$-estimate, \cite{Caf89, CC95, Lin86, Wan92a}) If $u \in S(\lambda, \Lambda, f)$ in $Q_1$, then there exists a universal $\delta>0$ such that $u \in W^{2, \delta}(Q_{1/2})$ with the uniform estimate
	\begin{align*}
		\|u\|_{W^{2, \delta}(Q_{1/2})} \leq C\left(\|u\|_{L^{\infty}(Q_1)}+\|f\|_{L^{n+1}(Q_1)}\right).
	\end{align*}
	
	\item (Interior $C^{\alpha}$-estimate,  \cite{CC95, KS79, KS80, Wan92a}) If $u \in S(\lambda, \Lambda, f)$ in $Q_1$, then $u$ is locally H\"older continuous in $Q_1$ with the uniform estimate
	\begin{align*}
		\|u\|_{C^{\alpha}(\overline{Q_{1/2}})} \leq C\left(\|u\|_{L^{\infty}(Q_1)}+\|f\|_{L^{n+1}(Q_1)}\right),
	\end{align*}
	where $\alpha \in (0, 1)$ and $C>0$ are universal constants which depend only on $n$, $\lambda$, and $\Lambda$. The interior H\"older regularity result essentially comes from Harnack inequality.
	
	\item (Counterexamples, \cite{NV07, NV08, NV10}) 
	We consider the elliptic situation with $n \geq 5$, and define the elliptic solution class $S_{\textnormal{ell}}(\lambda, \Lambda, 0)$ by the space of functions $u \in C(B_1)$ satisfying 
	\begin{equation*}
		\mathcal{M}^{-}_{\lambda, \Lambda}(D^2u) \leq 0 \leq \mathcal{M}^{+}_{\lambda, \Lambda}(D^2u)   \quad \text{in $B_1$}
	\end{equation*}
	in the viscosity sense. For any $\alpha \in (0,1)$, there exist some $(\lambda, \Lambda)$-elliptic fully nonlinear operator $F$ and a viscosity solution $u$ of $F(D^2u)=0$ such that $u$ does not belong to $C^{1, \alpha}(B_1)$. Since the difference quotient $(u(x+he)-u(x))/h$ for $h>0$ and $e \in \mathbb{R}^n$ with $|e|=1$ is contained in $S_{\textnormal{ell}}(\lambda, \Lambda, 0)$, we conclude that functions in $S_{\textnormal{ell}}(\lambda, \Lambda, 0)$ are in general only H\"older continuous, and this interior H\"older regularity cannot be improved, at least for $n \geq 5$.
	
	\item (Boundary $C^{1, \alpha}$-estimate, \cite{ABV20, LZ20, SS14, Wan92b}) Let $\Omega$ be a bounded domain with $(0,0) \in \partial_p \Omega$, and $u \in S(\lambda, \Lambda, f)$ in $\Omega \cap Q_1$ with $u=g$ on $\partial_p \Omega \cap Q_1$. We also let $\alpha \in (0,\overline{\alpha})$ for a universal constant $\overline{\alpha} \in (0,1)$. If $g\in C^{1,\alpha}(0,0)$ and $\partial_p \Omega \cap Q_1 \in C^{1, \alpha}(0,0)$, then $u\in C^{1,\alpha}(0,0)$ with the pointwise estimate
	\begin{equation*}
		|u(x,t)-L(x)|\leq C (|x|+\sqrt{|t|})^{1+\alpha} \quad \text{for all } (x,t) \in \Omega \cap Q_1.
	\end{equation*}
	The universal constant $\overline{\alpha}$ is strongly related to the boundary regularity of solutions to the particular model problem, which will be specified soon.
\end{enumerate}

Our main goal in this paper is to establish the global $C^{1, \alpha}$-estimate for viscosity solutions $u$ of \eqref{eq:s*_class}, provided that the ellipticity constants $\lambda$ and $\Lambda$ are close enough, i.e., $\Lambda/\lambda-1\ll 1$. Our first main theorem is concerned with the interior H\"older regularity for the gradient.
% We expect that this $C^{1, \alpha}$-regularity for functions in the solution class $S^{\ast}(\lambda, \Lambda, f)$ will attract independent interests, besides its application to the parabolic normalized $p$-Laplace equations.
\begin{theorem}[Interior $C^{1, \alpha}$-estimate]\label{thm:main1}
	Let $\alpha \in (0, 1)$ and $u$ satisfy \eqref{eq:s*_class} with $f \in C(Q_1) \cap  L^{\infty}(Q_1)$. Then there exists $\delta>0$ depending only on $n$ and $\alpha$ such that if $\Lambda \leq (1+\delta)\lambda $, then $u \in C_{\textnormal{loc}}^{1, \alpha}(Q_1)$ such that
	\begin{align*}
		\|u\|_{C^{1,\alpha}(\overline{Q_{1/2}})} \leq C\left(\|u\|_{L^{\infty}(Q_1)}+\|f\|_{L^{\infty}(Q_1)}\right) , 
	\end{align*}
	where $C>0$ depends only on $n$, $\lambda$, and $\alpha$.
\end{theorem}

Note that, in view of (ii) and (iii) above, we improved the interior regularity for functions in $S^{\ast}(\lambda, \Lambda, f)$ under the appropriate assumptions on the ellipticity constants. Furthermore, our regularity result for solution classes is optimal in the following sense:

 For $\delta >0$, let $\lambda=1$, $\Lambda=1+\delta$, and $f \equiv -2$. Then the function 
	\begin{equation*}
		u(x,t) =\begin{dcases}
			x_n^2 & (x_n<0) \\
			(1+\delta)^{-1} x_n^2 & (x_n\ge0) 
		\end{dcases}
	\end{equation*}
	belongs to $S(\lambda,\Lambda, f)$, but is not twice differentiable at $\{(x, t) :x_n=0\}$. Hence, we cannot expect that, in general, the function in the solution classes $S(\lambda, \Lambda, f)$ or $S^{\ast}(\lambda, \Lambda, f)$ to be more regular than $C^{1, 1}_{\mathrm{loc}}(Q_1)$, even though we further assume that $\Lambda$ is close enough to $\lambda$. In other words, \Cref{thm:main1} is the sharp regularity result that can be obtained for functions in the solution classes $S(\lambda, \Lambda, f)$ or $S^{\ast}(\lambda, \Lambda, f)$, when $\Lambda$ is close to $\lambda$.

We next state the boundary regularity result for functions in solution classes.
\begin{theorem}[Boundary $C^{1, \alpha}$-estimate]\label{thm:bdry}
	Let $\alpha \in (0, 1)$, $f\in C(Q_1) \cap L^{\infty}(Q_1)$, and $g\in C^{1,\alpha}(x_0,t_0)$ for some $(x_0,t_0) \in \partial_p Q_1$. Assume that $u$ satisfy 
	\begin{equation} \label{eq:s*_class_bd}
		\left\{\begin{aligned}
			u & \in S^*(\lambda,\Lambda, f) && \text{in } Q_1 \cap Q_1(x_0,t_0) \\
			u &= g && \text{on } \partial_p Q_1 \cap Q_1(x_0,t_0) .
		\end{aligned}\right.
	\end{equation}
	Then there exists $\delta>0$ depending only on $n$ and $\alpha$ such that if $\Lambda \leq (1+\delta)\lambda $, then $u \in C^{1, \alpha}(x_0,t_0)$, i.e., there exists a linear function $L$ such that
	\begin{align*}
		|u(x,t)-L(x)|\leq C(\|u\|_{L^{\infty}(Q_1)}+\|f\|_{L^{\infty}(Q_1)}+ \|g\|_{C^{1,\alpha}(x_0,t_0)})(|x-x_0|+\sqrt{|t-t_0|})^{1+\alpha} 
	\end{align*}
	for all $(x,t) \in Q_1 \cap Q_1(x_0,t_0)$ and 
	\begin{equation*}
		|Du(0)| \leq C(\|u\|_{L^{\infty}(Q_1)}+\|f\|_{L^{\infty}(Q_1)}+ \|g\|_{C^{1,\alpha}(x_0,t_0)}),
	\end{equation*}
	where $C>0$ depends only on $n$, $\lambda$, and $\alpha$.
\end{theorem}
We would like to illustrate our pointwise boundary estimate in view of the preceding literature such as \cite{ABV20, AS23, Kry83, LWZ20, LZ20, LZ22, SS14, Wan92b}. To be precise, we first consider the simple `model' problem: let $u$ belong to $S(\lambda, \Lambda, 0)$ in $Q_1 \cap \{x_n>0\}$ with a zero boundary condition on a flat boundary $Q_1 \cap \{x_n=0\}$. An application of Harnack inequality, Lipschitz estimate, and Hopf principle yields that there exists a universal constant $\overline{\alpha} \in (0,1)$ depending only on $n$, $\lambda$, and $\Lambda$ such that $u$ is $C^{1, \overline{\alpha}}$ at boundary points under some structural conditions. We remark that it is difficult to determine the precise value of $\overline{\alpha}$ due to the implicit construction of barrier functions in this step. Then by the compactness argument, the pointwise boundary $C^{1, \alpha}$-regularity result for the general case can be achieved whenever $\alpha \in (0, \overline{\alpha})$. In this context, our result can be interpreted as follows: the universal constant $\overline{\alpha}$ converges to $1$ when $\Lambda/\lambda-1$ goes to $0$.

We shortly display our strategy to prove \Cref{thm:main1} and \Cref{thm:bdry}. Indeed, the proofs of two theorems are based on the compactness argument together with the iteration method; see \cite{IS13, Tei15, Wan92b, WN23} for example. For the interior case, if $\|f\|_{L^{\infty}(Q_1)}$ is small enough and $\lambda$, $\Lambda$ are close enough, then we can expect the relation ``$S^{\ast}(\lambda, \Lambda, f) \approx S^{\ast}(\lambda, \lambda, 0)$". Since $u \in S^{\ast}(\lambda, \lambda, 0)$ if and only if $u_t-\lambda \Delta u=0$ in $Q_1$, we capture the smooth property of the limit function, and then show the approximation lemma which plays a crucial role in the iteration scheme. For the boundary case, we further take the $C^{1, \alpha}$-norm of $g$ and $\partial_p \Omega$ sufficiently small to arrive at the simple model problem.

We now move our attention to several applications of \Cref{thm:main1} and \Cref{thm:bdry}. We begin with the global estimates for functions in the solution class. In short, this global result can be derived from \Cref{thm:main1} and \Cref{thm:bdry} together with the machinery developed in \cite[Theorem 3.1]{Wan92b}.

\begin{theorem}[Global $C^{1, \alpha}$-estimate]\label{thm:global}
	Let $\alpha \in(0,1)$, $f\in C(Q_1) \cap L^{\infty}(Q_1)$, and $g\in C^{1,\alpha}(\partial_p Q_1)$. Assume that $u \in C(\overline{Q_1})$ satisfy 
		\begin{equation} \label{eq:diric}
		\left\{\begin{aligned}
			u & \in S^*(\lambda,\Lambda, f) && \text{in } Q_1 \\
			u &= g && \text{on } \partial_p Q_1.
		\end{aligned}\right.
	\end{equation}
	Then there exists $\delta>0$ depending only on $n$ and $\alpha$ such that if $\Lambda \leq (1+\delta)\lambda $, then $u \in C^{1, \alpha}(\overline{Q_1})$ with the uniform estimate
	\begin{align*}
		\|u\|_{C^{1, \alpha}(\overline{Q_1})} \leq C( \|u\|_{L^{\infty}(Q_1)} + \|f\|_{L^{\infty}(Q_1)} +  \|g\|_{C^{1,\alpha}(\partial_pQ_1)}),
	\end{align*}
	where $C>0$ is a constant depending only on $n$, $\lambda$, and $\alpha$.
\end{theorem}

Finally, such global $C^{1, \alpha}$-estimate for functions in solution classes can be applied to describe the regular property for viscosity solutions of the parabolic normalized $p$-Laplace equations given by
\begin{align}\label{eq:normalized}
	u_t=\Delta_p^Nu\coloneqq |Du|^{2-p}\,\text{div}(|Du|^{p-2}Du) ,
\end{align}
where $1<p<\infty$. The equation \eqref{eq:normalized} arises from tug-of-war-like stochastic games with noise; see \cite{ETT15, MPR10, PSSW09, PS08} and references therein. The existence, uniqueness, and regularity for viscosity solutions of \eqref{eq:normalized} has been widely studied in both elliptic and parabolic settings recently; we refer to \cite{AS22, Att20, APR17, AR20, BG13, BG15, Doe11, FZ21, HL20, JS17}. In particular, Jin and Silvestre \cite{JS17} developed the interior $C^{1, \alpha}$-regularity of $u$ for some $\alpha \in (0,1)$ and for any $p \in (1, \infty)$. On the other hand, Andrade and Santos proved \cite{AS22} the interior $C^{1, \alpha}$-regularity of $u$ for given $\alpha \in (0,1)$ and for $p$ close to 2; i.e., $|p-2|<\varepsilon=\varepsilon(n, \alpha)$. It is noteworthy that \Cref{thm:main1} together with the regularization scheme provides a new proof for \cite[Theorem 1.1]{AS22}; see \Cref{sec:normalized} for details.

An application of \Cref{thm:global} shows the global $C^{1, \alpha}$-regularity of viscosity solutions of \eqref{eq:normalized} for any $\alpha \in (0,1)$ when $p$ is close to 2.

\begin{theorem} [Parabolic normalized $p$-Laplace equations] \label{thm:normalized}
	Let $\alpha\in(0,1)$, $f\in C(Q_1) \cap L^{\infty}(Q_1)$, and $g\in C^{1,\alpha}(\overline{Q_1})$. Assume that $u \in C(\overline{Q_1)}$ is a viscosity solution of 
	\begin{equation*}
		\left\{\begin{aligned}
			u_t - \Delta_p^Nu &= f && \text{in } Q_1\\
			u&=g && \text{on } \partial_p Q_1 .
		\end{aligned}\right.
	\end{equation*}
	Then there exists $\delta>0$ depending only on $n$ and $\alpha$ such that if $|p-2|<\delta$, then $u \in C^{1, \alpha} (\overline{Q_1})$ with the uniform estimate
	\begin{align*}
		\|u\|_{C^{1, \alpha}(\overline{Q_1})} \leq C( \|u\|_{L^{\infty}(Q_1)} + \|f\|_{L^{\infty}(Q_1)} +  \|g\|_{C^{1,\alpha}(\overline{Q_1})}),
	\end{align*}
	where $C>0$ is a constant depending only on $n$, $p$, and $\alpha$.
\end{theorem}
When we investigate the parabolic normalized $p$-Laplace equations \eqref{eq:normalized}, the essential challenge arises from the fact that the coefficient matrix strongly depends on the gradient of viscosity solutions $u$. In particular, if we aim at the $C^{1, \alpha}$-estimate by finding a linear function $L$ which approximates $u$, then the key step is to iterate a similar argument for the difference function $u-L$. However, the quasilinear structure of \eqref{eq:normalized} leads to the translation of the original equation. We overcome such difficulty by understanding $u$ as functions in the solution class $S^{\ast}(\lambda, \Lambda, f)$ which is invariant under the replacement of $u$ by $u-L$. We also expect that this strategy can be widely employed to remove the dependence of $Du$ in certain equations, whenever the a priori Lipschitz estimate for $u$ is known.

It is noteworthy that our theorems \Cref{thm:main1}-\Cref{thm:normalized} still hold for the elliptic setting, after the essentially same analysis. Moreover, our approach for parabolic normalized $p$-Laplace equations is flexible enough so that it is also available for more general situations. For example, the global $C^{1, \alpha}$-regularity result is still valid for parabolic normalized $p(x, t)$-Laplace equations whenever $(x, t) \mapsto p(x, t)$ is Lipschitz and $|p(x, t)-2|$ is sufficiently small. We note that the proof is similar to the one for \Cref{thm:normalized} with small modifications; see \Cref{sec:normalized} for details. For related results on parabolic $\Delta_{p(x, t)}^N$ operators, we refer to the recent paper \cite{FZ21}, where they showed the interior $C^{1, \alpha}$-regularity for small $\alpha \in (0,1)$ provided that $p(x, t)$ belongs to $C^1$.

The paper is organized as follows. In \Cref{sec:preliminaries}, we summarize several notations and definitions which will be used throughout the paper. \Cref{sec:solutionclass} is devoted to the proof of \Cref{thm:main1} and \Cref{thm:bdry} for solution classes. Finally, we provide the applications of \Cref{thm:main1} and \Cref{thm:bdry}: global estimates for parabolic normalized $p$-Laplace equations.

\section{Preliminaries}\label{sec:preliminaries}
%==========================================================
%					Subsection: Notations
%==========================================================
In this section, we summarize some basic notations and gather definitions used throughout the paper. 

For a point $(x_0,t_0) \in \mathbb{R}^{n+1}$ and $r>0$, we denote the cylinder as
\begin{align*}
	Q_r(x_0,t_0) = B_r(x_0) \times (t_0 -r^2 , t_0] \quad \text{and} \quad
	Q_r^+(x_0,t_0) = \{(x,t) \in Q_r(x_0,t_0): x_n >0 \}.
\end{align*}
Moreover, we define the parabolic boundary as
\begin{align*}
	\partial_p Q_r(x_0,t_0) & = (\overline{B_r(x_0)} \times \{ t= t_0-r^2 \}) \cup(\partial B_r(x_0) \times [t_0 -r^2 , t_0)). 
\end{align*}
For convenience, we denote $Q_r = Q_r(0,0)$ and  $Q_r^+ = Q_r^+(0,0)$.

We denote partial derivatives of $u$ as subscriptions.
\begin{equation*}
	u_t = \partial_t u = \frac{\partial u}{\partial t} , \quad  D_iu =\frac{\partial u}{\partial x_i} , \quad \text{and} \quad  D_{ij} u = \frac{\partial^2 u}{\partial x_i \partial x_j} .
\end{equation*}

\subsection{Hölder spaces}
\begin{definition}[Hölder spaces]
	Let $\Omega \subset \mathbb{R}^{n+1}$ be an open set and $\alpha \in (0,1)$.
	\begin{enumerate}[(i)]
		\item $u \in C^{\alpha}(\overline{\Omega})$ means that there exists $C>0$ such that 
		$$|u(x,t)-u(y,s)| \leq C(|x-y|+\sqrt{|t-s|})^{\alpha} \quad \text{for all } (x,t), (y,s) \in \Omega.$$
		In other words, $u$ is $\frac{\alpha}{2}$-H\"older continuous in $t$ and  $\alpha$-H\"older continuous in $x$.
		\item $u \in C^{1,\alpha}(\overline{\Omega})$ means that $u$ is $\frac{\alpha + 1}{2}$-H\"older continuous in $t$ and $Du$ is $\alpha$-H\"older continuous in $x$.
	\end{enumerate}
\end{definition}

\begin{definition}
	Let $A \in \mathbb{R}^{n+1}$ be a bounded set and $f : A \to \mathbb{R}$ be a function. We say that $f$ is $C^{k,\alpha}$ at $(x_0,t_0) \in A$ (denoted by $f \in C^{k,\alpha}(x_0,t_0)$), if there exist constant $C>0$, $r>0$, and a polynomial $P(x,t)$ with $\deg P \leq k$ such that 
	\begin{equation}\label{cka_f}
		|f(x,t)-P(x,t)| \leq C(|x-x_0|+\sqrt{|t-t_0|})^{k+\alpha} \quad \text{for all } (x,t) \in A \cap Q_r(x_0,t_0).
	\end{equation}
	We define 
	\begin{align*}
		[f]_{C^{k,\alpha}(x_0,t_0)} &\coloneqq \inf \{C>0 : \eqref{cka_f} \text{ holds with } P(x,t) \text{ and } A \}, \\
		\|f\|_{C^{k,\alpha}(x_0,t_0)} &\coloneqq [f]_{C^{k,\alpha} (x_0,t_0)} + \sum_{i=0}^k \|D^i P(0,0)\|.
	\end{align*}
\end{definition}

\begin{definition}
	Let $\Omega$ be a bounded domain. We say that $\partial_p \Omega$ is $C^{k,\alpha}$ at $(x_0,t_0) \in \partial_p \Omega$ (denoted by $\partial_p \Omega \in C^{k,\alpha}(x_0,t_0)$), if there exist constant $C>0$, $r>0$, a new coordinate system $\{x_1,\cdots,x_n,t\}$, and a polynomial $P(x',t)$ with $\deg P \leq k$, $P(0,0)=0$, $DP(0,0)=0$ such that $(x_0,t_0)=(0,0)$ in this coordinate system,
	\begin{equation}\label{cka_dom}
		\begin{aligned}
			&Q_\rho \cap \{ (x',x_n,t):x_n > P(x',t) + C(|x'|+\sqrt{|t|})^{k+\alpha} \} \subset Q_r \cap \Omega, \\
			&Q_\rho \cap \{ (x',x_n,t):x_n < P(x',t) - C(|x'|+\sqrt{|t|})^{k+\alpha} \} \subset Q_r \cap \Omega^c.
		\end{aligned}
	\end{equation}
	We define 
	\begin{align*}
		[\partial_p \Omega]_{C^{k,\alpha}(x_0,t_0)} &\coloneqq \inf \{C>0 : \eqref{cka_dom} \text{ holds with } P(x',t) \text{ and } C \}, \\
		\|\partial_p \Omega\|_{C^{k,\alpha}(x_0,t_0)} &\coloneqq [\partial_p \Omega]_{C^{k,\alpha} (x_0,t_0)}  + \sum_{i=2}^k \|D^i P(0',0)\|.
	\end{align*}
\end{definition}

%==========================================================
%		Subsection: Definitions(Viscosity Solutions)
%==========================================================
\subsection{Viscosity solutions and solution classes}
\begin{definition} [Test functions]
	Let $u$ be a continuous function in $Q_1 $. The function $\varphi : Q_1 \to \mathbb{R}$ is called \textit{test function} if it is $C^1$ with respect to $t$ and $C^2$ with respect to $x$.
	\begin{enumerate} [(i)]
		\item We say that the test function $\varphi$ touches $u$ from above at $(x,t)$ if there exists an open neighborhood $U$ of $(x,t)$ such that 
		$$u \leq \varphi  \quad \mbox{in } U \qquad  \mbox{and} \qquad u(x,t) = \varphi(x,t). $$
		\item We say that the test function $\varphi$ touches $u$ from below at $(x,t)$ if there exists an open neighborhood $U$ of $(x,t)$ such that 
		$$u \ge \varphi  \quad \mbox{in } U \qquad  \mbox{and} \qquad u(x,t) = \varphi(x,t). $$
	\end{enumerate}
\end{definition}

\begin{definition}[Viscosity solutions] \label{vis1}
	Let $u$ be a function defined in $Q_1$.
	\begin{enumerate} [(i)]
		\item Let $u$ be a upper semicontinuous function in $Q_1$. $u$ is called a \textit{viscosity subsoution} of \eqref{eq:n_p_lap} in $Q_1 $ when the following condition holds: if for any $(x,t) \in Q_1 $ and any test function $\varphi$ touching $u$ from above at $(x,t)$, then
		\begin{equation*}
			\varphi_t (x,t) - F ( D^2 \varphi (x,t),x,t ) \leq f(x,t).
		\end{equation*}
		\item Let $u$ be a lower semicontinuous function in $Q_1$. $u$ is called a \textit{viscosity supersoution} of \eqref{eq:n_p_lap} in $Q_1 $ when the following condition holds: if for any $(x,t) \in Q_1 $ and any test function $\varphi$ touching $u$ from below at $(x,t)$, then
		\begin{equation*}
			\varphi_t (x,t) - F ( D^2 \varphi (x,t),x,t ) \ge f(x,t).
		\end{equation*}
	\end{enumerate}
\end{definition}

\begin{definition}[Solution classes]
	We define several solutions classes as follows:
	\begin{align*}
		\underline{S}(\lambda, \Lambda,f ) &\coloneqq  \{u \mid u_t - \mathcal{M}^{+}_{\lambda, \Lambda}(D^2u) \leq f \ \text{in the viscosity sense}  \}, \\
		\overline{S}(\lambda, \Lambda,f ) &\coloneqq   \{u \mid u_t - \mathcal{M}^{-}_{\lambda, \Lambda}(D^2u) \ge f  \ \text{in the viscosity sense}   \}, \\
		S(\lambda,\Lambda, f) &\coloneqq \underline{S}(\lambda, \Lambda,f )  \cap \overline{S}(\lambda, \Lambda,f ), \\
		S^*(\lambda,\Lambda, f) &\coloneqq \underline{S}(\lambda, \Lambda, \|f\|_{L^{\infty}})  \cap \overline{S}(\lambda, \Lambda,-\|f\|_{L^{\infty}}).  
	\end{align*}
	In particular, $S(\lambda, \Lambda, f)$ is contained in $S^{\ast}(\lambda, \Lambda, f)$.
\end{definition}

We next introduce the definition of viscosity solutions of parabolic normalized $p$-Laplace equations
\begin{equation} \label{eq:n_p_lap}
	u_t-\Delta_p^Nu=f \quad \text{in $Q_1$},
\end{equation}
where $p \in (1,\infty)$. We observe that
\begin{align*}
	\Delta_p^Nu&=|Du|^{2-p} \, \mathrm{div}(|Du|^{p-2}Du)=(\delta_{ij}+(p-2)|Du|^{-2}u_iu_j)u_{ij} \eqqcolon a_{ij}(Du) u_{ij}
\end{align*}
and
\begin{align*}
	\min\{p-1, 1\} I \leq \delta_{ij}+(p-2)q_iq_j|q|^{-2} \leq \max\{p-1, 1\}I \quad \text{for all $q \in \mathbb{R}^n \setminus \{0\}$.}
\end{align*}
Note that \eqref{eq:n_p_lap} is undeﬁned when $Du = 0$, where it has bounded discontinuous coefficients $a_{ij}(Du)$. This can be resolved by adapting the notion of viscosity solution using the upper and lower semicontinuous envelopes of the normalized $p$-Laplacian.

We denote the maximum and the minimum eigenvalues associated with $M \in \mathcal{S}^n$ as follows:
\begin{equation*}
	e_{\textnormal{min}}(M) \coloneqq \min_{|e|=1} (Me \cdot e) \quad \text{and} \quad e_{\textnormal{max}}(M) \coloneqq \max_{|e|=1} (Me \cdot e).
\end{equation*}

\begin{definition} \label{vis2}
	Let $u$ be a function defined in $Q_1$. 
	\begin{enumerate} [(i)]
		\item Let $u$ be a upper semicontinuous function in $Q_1$. $u$ is called a \textit{viscosity subsoution} of \eqref{eq:n_p_lap} in $Q_1 $ when the following condition holds: if for any $(x,t) \in Q_1 $ and any test function $\varphi$ touching $u$ from above at $(x,t)$, then
		\begin{equation*}
			\left\{\begin{aligned}
				\varphi_t (x,t) - \Delta_p^N \varphi (x,t) &\leq f(x,t) && \text{if } D\varphi(x,t) \ne 0 \\
				\varphi_t (x,t) - \Delta \varphi (x,t) -(p-2)e_{\textnormal{max}}(D^2 \varphi(x,t)) &\leq f(x,t) && \text{if } D\varphi(x,t) \ne 0, \, p \ge 2 \\
				\varphi_t (x,t) - \Delta \varphi (x,t) -(p-2)e_{\textnormal{min}}(D^2 \varphi(x,t)) &\leq f(x,t) && \text{if } D\varphi(x,t) \ne 0, \, 1 < p < 2. 
			\end{aligned}\right.
		\end{equation*}
		\item Let $u$ be a lower semicontinuous function in $Q_1$. $u$ is called a \textit{viscosity supersoution} of \eqref{eq:n_p_lap} in $Q_1 $ when the following condition holds: if for any $(x,t) \in Q_1 $ and any test function $\varphi$ touching $u$ from below at $(x,t)$, then
		\begin{equation*}
			\left\{\begin{aligned}
				\varphi_t (x,t) - \Delta_p^N \varphi (x,t) &\ge f(x,t) && \text{if } D\varphi(x,t) \ne 0 \\
				\varphi_t (x,t) - \Delta \varphi (x,t) -(p-2)e_{\textnormal{min}}(D^2 \varphi(x,t)) &\ge f(x,t) && \text{if } D\varphi(x,t) \ne 0, \, p \ge 2 \\
				\varphi_t (x,t) - \Delta \varphi (x,t) -(p-2)e_{\textnormal{max}}(D^2 \varphi(x,t)) &\ge f(x,t) && \text{if } D\varphi(x,t) \ne 0, \, 1 < p < 2. 
			\end{aligned}\right.
		\end{equation*}
	\end{enumerate}
\end{definition}

%==========================================================
%		Subsection: Definitions(Solutions Classes)
%==========================================================

\section{Regularity for functions in solution classes}\label{sec:solutionclass}
\subsection{Interior $C^{1,\alpha}$-regularity}
We begin with the following approximation lemma.
\begin{lemma}\label{lem:approx_c1a}
	For any $\alpha \in (0,1)$, there exists $\delta>0$ depending only on $n$ and $\alpha$ such that if $u$ satisfies \eqref{eq:s*_class} with 
	\begin{equation*}
	\|u\|_{L^{\infty}(Q_1)} \leq 1, \quad \|f\|_{L^{\infty}(Q_1)} \leq \delta, \quad \text{and} \quad \Lambda \leq (1+\delta)\lambda, 
	\end{equation*}
	then there exists a linear function $L$ such that
	\begin{align*}
		\|u-L\|_{L^{\infty}(Q_{\eta})} \leq \eta^{1+\alpha}
	\end{align*}
and
\begin{align*}
	|L(0,0)|+|DL| \leq C_1,
\end{align*}
where $C_1>1$ depends only on $n$ and $\lambda$, and $\eta \in(0,1)$ depends only on $n$, $\lambda$, and $\alpha$.
\end{lemma}

\begin{proof}
	We prove by contradiction; we suppose the conclusion does not hold. In other words, there exist sequences $u_k \in C(Q_1)$, $f_k \in C(Q_1) \cap L^{\infty}(Q_1)$,  and $\Lambda_k \geq \lambda$ such that $u_k$ satisfies
\begin{equation*}
	\left\{\begin{aligned}
		\partial_t u_k - \mathcal{M}^{-}_{\lambda, \Lambda_k}(D^2u_k) &\ge -\|f_k\|_{L^{\infty}(Q_1)} \\
		\partial_t u_k - \mathcal{M}^{+}_{\lambda, \Lambda_k}(D^2u_k) & \leq \|f_k\|_{L^{\infty}(Q_1)} 
	\end{aligned}\right.
	\quad \text{in $Q_1$},
\end{equation*}
	\begin{align*}
		\|u_k\|_{L^{\infty}(Q_1)} \leq 1, \quad \|f_k\|_{L^{\infty}(Q_1)} \leq 1/k, \quad \text{and} \quad \Lambda_k \leq (1+ 1/k)\lambda.
	\end{align*}
Moreover, for any linear function $L$ satisfying $|L(0,0)|+|DL| \leq C_1$, we have
\begin{align}\label{eq:contra_int}
	\|u_k-L\|_{L^{\infty}(Q_{\eta})} >\eta^{1+\alpha},
\end{align}
where $C_1>0$ and $\eta \in (0,1)$ will be determined later.

We first note that $\{u_k\}_{k=1}^{\infty}$ is uniformly bounded in $L^{\infty}(Q_1)$ and by the interior H\"older estimate (\cite[Theorem 4.19]{Wan92a}), $\{u_k\}_{k=1}^{\infty}$ is equicontinuous in compact sets of $Q_1$. Thus, by Arzela-Ascoli theorem, there exist a subsequence $\{u_{k_j}\}_{j=1}^{\infty}$ of $\{u_k\}_{k=1}^{\infty}$ and a limit function $\overline{u} \in C(Q_1)$ such that $u_{k_j} \to \overline{u}$ uniformly in compact sets of $Q_1$ as $j \to \infty$. Moreover, since 
\begin{equation*} 
	\mathcal{M}_{\lambda, \Lambda_{k_j}}^{\pm}(M) \to \lambda \, \mathrm{tr}(M) \text{ uniformly in compact sets of }\mathcal{S}^n \text{ as } j \to \infty
\end{equation*}
and  $\|f_{k_j}\|_{L^{\infty}(Q_1)} \to 0$ as $j \to \infty$, the stability theorem yields that 
\begin{align*}
	\overline{u}_t - \lambda \Delta \overline{u} =0 \quad \text{in $Q_1$}.
\end{align*} 
We now choose a linear function $\overline{L}$ by 
\begin{align*}
	\overline{L}(x) \coloneqq \overline{u}(0,0)+D\overline{u}(0,0) \cdot x .
\end{align*}
Then, by the interior estimates for the uniformly parabolic equations, there exists $C>0$ that depends only on $n$ and $\lambda$ such that
\begin{align*}
	|\overline{u}(x,t)-\overline{L}(x)| \leq C(|x|+\sqrt{|t|})^2 \quad \text{for any $(x,t) \in Q_{1/2}$}
\end{align*}
and $|\overline{u}(0,0)|+|D\overline{u}(0,0)| \leq C$.
We take $C_1 \coloneqq 1+C$ and $\eta$ small enough so that $8C\eta^{1-\alpha} <1/2$. Then we observe that
\begin{align*}
	\|\overline{u}-\overline{L}\|_{L^{\infty}(Q_{\eta})} \leq 8C \eta^2 <\eta^{1+\alpha}/2.
\end{align*}
On the other hand, by letting $k_j \to \infty$ in \eqref{eq:contra_int}, we have
\begin{align*}
	\|\overline{u}-\overline{L}\|_{L^{\infty}(Q_{\eta})} \ge\eta^{1+\alpha},
\end{align*}
which leads to the contradiction.
\end{proof}

We would like to iterate the result in \Cref{lem:approx_c1a} to find a sequence of linear functions $\{L_k\}_{k=0}^{\infty}$ close to $u$ in a discrete way.

\begin{theorem} \label{thm:u-Lk}
	Let $\delta>0$ be as in \Cref{lem:approx_c1a}, and $u$ satisfy \eqref{eq:s*_class}. Suppose that
	\begin{equation*}
		\|u\|_{L^{\infty}(Q_1)} \leq 1, \quad \|f\|_{L^{\infty}(Q_1)} \leq \delta, \quad \text{and} \quad\Lambda \leq (1+ \delta)\lambda.
	\end{equation*}
Then there exists a sequence of linear functions
\begin{align*}
	L_k(x)=a_k+b_k\cdot x, \quad k \geq -1
\end{align*}
such that for all $k \geq 0$, we have 
\begin{align}\label{eq:lin1}
	\|u-L_k\|_{L^{\infty}(Q_{\eta^k})} \leq \eta^{k(1+\alpha)},\
\end{align}
and
\begin{align}\label{eq:lin2}
	|a_k-a_{k-1}|+\eta^{k} |b_k-b_{k-1}| \leq 2C_1 \eta^{(k-1)(1+\alpha)},
\end{align}
where $C_1>1$ and $\eta \in (0,1)$ are determined in \Cref{lem:approx_c1a}.
\end{theorem}

\begin{proof}
	We argue by induction. For $k=0$, by setting $L_{-1}=L_0=0$, the conditions immediately hold. Suppose that the conclusion holds for $k \geq 0$. We claim that the conclusion also holds for $k+1$.
	
	For this purpose, let $r=\eta^k$, $y=x/r$, $s=t/r^2$ and 
	\begin{align*}
		v(y,s) \coloneqq \frac{u(x,t)-L_k(x)}{r^{1+\alpha}}.
	\end{align*}
Then we observe that for $(y,s) \in Q_1$,
\begin{align*}
	\partial_s v(y,s) - \mathcal{M}_{\lambda, \Lambda}^+(D_y^2v(y,s)) &=r^{1-\alpha}\left( u_t(x,t) -\mathcal{M}_{\lambda, \Lambda}^+(D^2u(x,t))\right)\\
	& \leq  r^{1-\alpha} \|f\|_{L^{\infty}(Q_r)} .
\end{align*}
A similar calculation for $\mathcal{M}^-_{\lambda, \Lambda}(D^2v)$ shows that $v \in S^{\ast}(\lambda, \Lambda, \widetilde{f})$ in $Q_1$ with 
\begin{equation*}
	\|v\|_{L^{\infty}(Q_1)} \leq 1, \quad \|\widetilde{f}\|_{L^{\infty}(Q_1)} \leq \delta, \quad \text{and} \quad \Lambda\leq(1+\delta)\lambda,
\end{equation*}
where $\widetilde{f}(y,s) \coloneqq r^{1-\alpha} f(x,t)$. Thus, by applying \Cref{lem:approx_c1a} to $v$, there exists a linear function $\widetilde{L}$ such that 
\begin{align}\label{eq:induc1}
	\|v-\widetilde{L}\|_{L^{\infty}(Q_{\eta})} \leq \eta^{1+\alpha}
\end{align}
and
\begin{align}\label{eq:induc2}
	|\widetilde{L}(0,0)|+|D\widetilde{L}| \leq C_1.
\end{align}
We now let $L_{k+1}(x) \coloneqq L_k(x)+r^{1+\alpha} \widetilde{L}(x/r)$. Then \eqref{eq:induc2} implies that
\begin{align*}
	|a_{k+1}-a_k|+\eta^{k+1}|b_{k+1}-b_k| \leq 2C_1 \eta^{k(1+\alpha)}.
\end{align*}
Finally, \eqref{eq:induc1} shows that  
\begin{align*}
	|u(x,t)-L_{k+1}(x)| &=|u(x,t)-L_k(x)-r^{1+\alpha}\widetilde{L}(x/r) | \\
	&\leq r^{1+\alpha} |v(x/r,t/r^2)-\widetilde{L}(x/r)| \\
	&\leq \eta^{(k+1)(1+\alpha)} \quad \text{for all } (x,t) \in Q_{\eta^{k+1}},
\end{align*}
as desired.
\end{proof}

We are now ready to prove our first main theorem, \Cref{thm:main1}.

\begin{proof}[Proof of \Cref{thm:main1}]
	Let $u$ satisfy \eqref{eq:s*_class} with $f \in C(Q_1) \cap L^{\infty}(Q_1)$ and $\Lambda \leq (1+\delta)\lambda$ for $\delta>0$ which was chosen in \Cref{thm:u-Lk}. We claim that the assumptions in \Cref{thm:u-Lk} hold after the appropriate reduction argument. For $K\coloneqq \|u\|_{L^{\infty}(Q_1)}+\|f\|_{L^{\infty}(Q_1)}/\delta+1$, by considering  
	\begin{align*}
		\widetilde{u}(x,t) \coloneqq u(x,t)/K \quad \text{and} \quad \widetilde{f}(x,t) \coloneqq f(x,t)/K,
	\end{align*}
we may assume that 
\begin{equation*}
\|u\|_{L^{\infty}(Q_1)} \leq 1 \quad \text{and} \quad \|f\|_{L^{\infty}(Q_1)} \leq \delta.
\end{equation*}

Thus, we can apply \Cref{thm:u-Lk} for $u$ to find a sequence of linear function $L_k(x)=a_k+b_k \cdot x$ satisfying \eqref{eq:lin1} and \eqref{eq:lin2}. Indeed, by standard argument, we obtain limits $a$ and $b$ such that $a_k \to a$ and $b_k \to b$ satisfying
\begin{align*}
	|a_k-a| \leq C\eta^{k(1+\alpha)} \quad \text{and} \quad |b_k-b| \leq C\eta^{k\alpha}.
\end{align*}
Moreover, if we write $L(x)=a+b\cdot x$, then we have $L_k \to L$ and 
\begin{align*}
	\|L_k-L\|_{L^{\infty}(Q_{\eta^k})} \leq C\eta^{k(1+\alpha)}.
\end{align*}
Finally, for any $(x,t) \in Q_1$, there exists $j \geq 0$ such that $\eta^{j+1} \leq \max\{ |x|, \sqrt{|t|} \}< \eta^j$. Then we conclude that
\begin{align*}
	|u(x,t)-L(x)| \leq |u(x,t)-L_j(x)|+|L_j(x)-L(x)| \leq C(|x|+\sqrt{|t|})^{1+\alpha},
\end{align*}
which implies that $u$ is $C^{1, \alpha}$ at $(0,0)$.
\end{proof}

\subsection{Pointwise boundary $C^{1,\alpha}$-regularity}
In this section, we establish the pointwise boundary $C^{1,\alpha}$-estimate for $ u \in S^*(\lambda,\Lambda,f)$ in the more generalized situation; namely, we replace the unit cube $Q_1$ by a domain $\Omega$ which is $C^{1, \alpha}$ at one point. For this purpose, we assume that $(0,0) \in \partial_p \Omega$ and $\partial_p \Omega  \in C^{1,\alpha}(0,0)$, and define
\begin{equation*}
	\osc_{Q_r} \partial_p\Omega \coloneqq \sup_{(x,t) \in \partial_p\Omega\cap Q_r} x_n -\inf_{(x,t) \in \partial_p\Omega\cap Q_r} x_n \quad \text{for } r >0.
\end{equation*}

The following lemma corresponds to the approximation lemma at boundary points.
\begin{lemma}\label{lem:approx_c1a_bd}
	For any $\alpha \in (0,1)$, there exists $\delta>0$ depending only on $n$ and $\alpha$ such that if $u$ satisfies
	\begin{equation} \label{eq:dirichlet}
		\left\{\begin{aligned}
			u & \in S^*(\lambda,\Lambda, f) && \text{in } \Omega \cap Q_1 \\
			u &= g && \text{on } \partial_p \Omega \cap Q_1.
		\end{aligned}\right.
	\end{equation}
	with $\|u\|_{L^{\infty}(\Omega \cap Q_1)} \leq 1$, $\|f\|_{L^{\infty}(\Omega \cap Q_1)} \leq \delta$, $\|g\|_{L^{\infty}(\partial_p \Omega \cap Q_1)} \leq \delta$, $\osc_{Q_1} \partial_p \Omega \leq \delta$, and $\Lambda \leq (1+\delta)\lambda$, then there exists a constant $a\in\mathbb{R}$ such that
	\begin{align*}
		\|u-ax_n\|_{L^{\infty}(\Omega\cap Q_{\eta})} \leq \eta^{1+\alpha}
	\end{align*}
and
\begin{align*}
	|a| \leq C_1,
\end{align*}
where $C_1>1$ depends only on $n$ and $\lambda$, and $\eta \in(0,1)$ depends only on $n$, $\lambda$, and $\alpha$.
\end{lemma}

\begin{proof}
	We prove by contradiction; we suppose the conclusion does not hold. In other words, there exist sequences $u_k$, $f_k$, $g_k$, $\Omega_k$, and $\Lambda_k \geq \lambda$ such that $u_k$ satisfies
	\begin{equation*}
		\left\{\begin{aligned}
			u_k & \in S^*(\lambda,\Lambda_k, f_k) && \text{in } \Omega_k \cap Q_1 \\
			u_k &= g_k && \text{on } \partial_p \Omega_k \cap Q_1,
		\end{aligned}\right.
	\end{equation*}
$\|u_k\|_{L^{\infty}(\Omega_k \cap Q_1)} \leq 1$, $\|f_k\|_{L^{\infty}(\Omega_k \cap Q_1)} \leq 1/k$, $\|g_k\|_{L^{\infty}(\partial_p \Omega_k \cap Q_1)} \leq 1/k$, $\osc_{Q_1} \partial_p \Omega_k \leq 1/k$, and $\Lambda_k \leq (1+ 1/k)\lambda$. Moreover, for any constant $a \in \mathbb{R}$ satisfying $|a| \leq C_1$, we have
\begin{align}\label{eq:contradiction}
	\|u_k-ax_n\|_{L^{\infty}(\Omega \cap Q_{\eta})} >\eta^{1+\alpha},
\end{align}
where $C_1>0$ and $\eta \in (0,1)$ will be determined later.

As in the proof of \Cref{lem:approx_c1a}, there exist a subsequence $\{u_{k_j}\}_{j=1}^{\infty}$ of $\{u_k\}_{k=1}^{\infty}$ and a limit function $\overline{u} \in C(Q_1^+)$ such that $u_{k_j} \to \overline{u}$ uniformly in compact sets of $Q_1^+$ as $j \to \infty$. Moreover, $\overline{u}$ satisfies
\begin{equation*}
	\left\{\begin{aligned}
		\overline{u}_t - \lambda \Delta \overline{u} &=0 && \text{in $Q_1^+$} \\
		\overline{u} &=0 && \text{on $\{x_n=0\} \cap Q_1$}.
	\end{aligned}\right.
\end{equation*} 
By the reflection principle for the heat equation, the function 
\begin{equation*}
	\widetilde{u}(x,t) \coloneqq 
	\begin{cases}
		\overline{u}(x,t) & \text{if } x_n \ge 0 \\
		-\overline{u}(x',-x_n,t) & \text{if } x_n < 0,
	\end{cases}
\end{equation*}
satisfies $\widetilde{u}_t - \lambda \Delta \widetilde{u} =0$ in $Q_1$. Then, by the interior estimates for heat equations, there exists $C>0$ that depends only on $n$ and $\lambda$ such that
\begin{align*}
	|\widetilde{u}(x,t)- D_n\widetilde{u}(0,0) x_n | \leq C(|x|+\sqrt{|t|})^2 \quad \text{for any $(x,t) \in Q_{1/2}$}
\end{align*}
and $|D_n\overline{u}(0,0)|=|D_n\widetilde{u}(0,0)| \leq C$. We take $C_1 \coloneqq 1+C$ and $\eta$ small enough so that $8C\eta^{1-\alpha} <1/2$. Then we observe that
\begin{align*}
	\|\overline{u}-D_n\overline{u}(0,0) x_n\|_{L^{\infty}(Q_{\eta}^+)} \leq 8C \eta^2 <\eta^{1+\alpha}/2.
\end{align*}
On the other hand, by letting $k_j \to \infty$ in \eqref{eq:contradiction}, we have
\begin{align*}
	\|\overline{u}-D_n\overline{u}(0,0) x_n\|_{L^{\infty}(Q_{\eta}^+)} \ge\eta^{1+\alpha},
\end{align*}
which leads to the contradiction.
\end{proof}

\begin{theorem} \label{thm:u-Lk_bd}
	Let $\delta>0$, $\eta \in (0,1)$, and $C_1>1$ be as in \Cref{lem:approx_c1a_bd}, and $u$ satisfy \eqref{eq:dirichlet}. Suppose that
	\begin{equation*}
		\|u\|_{L^{\infty}(\Omega \cap Q_1)} \leq 1, \quad \|f\|_{L^{\infty}(\Omega \cap Q_1)} \leq \delta, \quad \| \partial_p \Omega \|_{C^{1,\alpha}(0,0)} \leq \frac{1-\eta^\alpha}{2C_1}\delta, \quad\Lambda \leq (1+ \delta)\lambda,
	\end{equation*}
	and
	\begin{equation*}  
		|g(x,t)| \leq \frac{\delta}{2^{2+\alpha}} (|x|+\sqrt{|t|})^{1+\alpha} \quad \text{for any $(x,t) \in \partial_p\Omega\cap Q_1$},
	\end{equation*}
Then there exists a sequence $\{a_k\}_{k=-1}^{\infty}$
such that for all $k \geq 0$, we have 
\begin{align}\label{eq:lin1_bd}
	\|u-a_kx_n\|_{L^{\infty}(\Omega \cap Q_{\eta^k})} \leq \eta^{k(1+\alpha)},\
\end{align}
and
\begin{align}\label{eq:lin2_bd}
	|a_k-a_{k-1}| \leq  C_1 \eta^{(k-1) \alpha}.
\end{align}
\end{theorem}

\begin{proof}
	We argue by induction. For $k=0$, by setting $a_{-1}=a_0=0$, the conditions immediately hold. Suppose that the conclusion holds for $k \geq 0$. We claim that the conclusion also holds for $k+1$.
	
	For this purpose, let $r=\eta^k$, $y=x/r$, $s=t/r^2$, and 
	\begin{align*}
		v(y,s) \coloneqq \frac{u(x,t)-a_kx_n}{r^{1+\alpha}}.
	\end{align*}
Then $v$ satisfies
	\begin{equation*}
		\left\{\begin{aligned}
			u & \in S^*(\lambda,\Lambda, \widetilde{f}) && \text{in } \widetilde{\Omega}\cap Q_1 \\
			u &= \widetilde{g} && \text{on } \partial_p\widetilde{\Omega}\cap Q_1,
		\end{aligned}\right.
	\end{equation*}
where
\begin{equation*}
	\widetilde{f}(y,s)\coloneqq r^{1-\alpha}f(x,t), \quad \widetilde{g}(y,s)\coloneqq \frac{g(x,t)-a_kx_n}{r^{1+\alpha}}, \quad \text{and} \quad \widetilde{\Omega}\coloneqq \{(y,s): (ry,r^2s) \in \Omega\}.
\end{equation*}
Then it immediately follows that $\|v\|_{L^{\infty}(\widetilde{\Omega} \cap Q_1)} \leq 1$ and $\|\widetilde{f}\|_{L^{\infty}(\widetilde{\Omega} \cap Q_1)} \leq \delta$. By \eqref{eq:lin2_bd}, we have
\begin{equation*}
	|a_k| \leq \sum_{i=1}^k |a_i -a_{i-1}| \leq \frac{C_1}{1-\eta^\alpha}
\end{equation*}
and hence 
\begin{align*}
	 |\widetilde{g}(y,s)| & \leq \frac{1}{r^{1+\alpha}}\left(\frac{\delta}{2^{2+\alpha}} \cdot 2^{1+\alpha}r^{1+\alpha} + \frac{C_1}{1-\eta^\alpha} \cdot \frac{1-\eta^\alpha}{2C_1}\delta r^{1+\alpha}\right) = \delta \quad \text{for all } (y,s) \in \partial_p \widetilde{\Omega} \cap Q_1.
\end{align*}
Furthermore, 
\begin{equation*}
	\osc_{Q_1} \partial_p  \widetilde{\Omega} = \frac{1}{r} \osc_{Q_r} \partial_p \Omega \leq 2\|\partial_p\Omega\|_{C^{1,\alpha}(0,0)} r^{\alpha} \leq 2 r^{\alpha} \cdot \frac{1-\eta^\alpha}{2C_1}\delta \leq \delta.
\end{equation*}
Thus, by applying \Cref{lem:approx_c1a_bd} to $v$, there exists a constant $\widetilde{a}$ such that 
\begin{align}\label{eq:induc1_bd}
	\|v-\widetilde{a}x_n\|_{L^{\infty}(\Omega \cap Q_{\eta})} \leq \eta^{1+\alpha}
\end{align}
and
\begin{align}\label{eq:induc2_bd}
	|\widetilde{a}| \leq C_1.
\end{align}
We now let $a_{k+1} \coloneqq a_k +r^{\alpha} \widetilde{a}$. Then \eqref{eq:induc2_bd} implies that
\begin{align*}
	|a_{k+1}-a_k|  \leq  C_1 \eta^{k\alpha}.
\end{align*}
Finally, \eqref{eq:induc1_bd} shows that  
\begin{align*}
	|u(x,t)-a_{k+1}x_n| &=|u(x,t)-a_kx_n-r^{\alpha}\widetilde{a}x_n | \\
	&\leq r^{1+\alpha} |v(x/r,t/r^2)-\widetilde{a}(x_n/r)| \\
	&\leq \eta^{(k+1)(1+\alpha)} \quad \text{for all } (x,t) \in Q_{\Omega \cap \eta^{k+1}},
\end{align*}
as desired.
\end{proof}
We are now ready to prove our second main theorem, \Cref{thm:bdry}.

\begin{proof}[Proof of \Cref{thm:bdry}]
	Without loss of generality, we can assume that $(x_0,t_0)=(0,0)$ and $\Omega$ is a cylinder whose lateral boundary contains the origin. Let $u$ satisfy \eqref{eq:s*_class_bd} with $f \in C(\Omega) \cap L^{\infty}(\Omega)$, $g\in C^{1,\alpha}(0,0)$, and $\Lambda \leq (1+\delta)\lambda$ for $\delta>0$ which was chosen in \Cref{thm:u-Lk_bd}. We claim that the assumptions in \Cref{thm:u-Lk_bd} hold after the appropriate reduction argument. Since $g\in C^{1,\alpha}(0,0)$, there exists a linear function $L_g$ such that
	\begin{equation*}
		|g(x,t)-L_g(x)| \leq [g]_{C^{1,\alpha}(0,0)} (|x|+\sqrt{|t|})^{1+\alpha} \quad \text{for all }(x,t) \in \Omega \cap Q_1.
	\end{equation*}
	Then $\overline{u} \coloneqq u-L_g$ satisfies 
	\begin{equation*} 
		\left\{\begin{aligned}
			\overline{u} & \in S^*(\lambda,\Lambda, f) && \text{in } \Omega \cap Q_1  \\
			\overline{u} &= \overline{g} && \text{on } \partial_p \Omega \cap Q_1 ,
		\end{aligned}\right.
	\end{equation*}
	where $\overline{g} \coloneqq g - L_g$ satisfies
	\begin{equation*}
		|\overline{g}(x,t)| \leq [g]_{C^{1,\alpha}(0,0)} (|x|+\sqrt{|t|})^{1+\alpha} \quad \text{for all }(x,t) \in \partial_p\Omega \cap Q_1 .
	\end{equation*}
	We next let $y=x/\rho$, $s=t/\rho^2$, and $\widehat{u}(y,s)=\overline{u}(x,t)$. Then $\widehat{u}$ satisfies 
	\begin{equation*} 
		\left\{\begin{aligned}
			\widehat{u} & \in S^*(\lambda,\Lambda, \widehat{f}) && \text{in } \widehat{\Omega} \cap Q_1 \\
			\widehat{u} &= \widehat{g} && \text{on } \partial_p \widehat{\Omega} \cap Q_1  ,
		\end{aligned}\right.
	\end{equation*}	
	where $\widehat{f}(y,s) \coloneqq \rho^2 f(x,t)$, $\widehat{g}(y,s) \coloneqq \overline{g}(x,t)$, and $\widehat{\Omega}\coloneqq \{(y,s): (\rho y,\rho^2s) \in \Omega\}$. Since $\Omega$ has a smooth lateral boundary, we have
	\begin{align*}
		|x_n| \leq [\partial_p \Omega]_{C^{2}(0,0)}(|x'|^2 + \sqrt{|t|})^2 \quad \text{for all } (x,t) \in \partial_p \Omega \cap Q_1.
	\end{align*}
	This implies that 
	\begin{align*}
		|y_n| \leq \rho [\partial_p \Omega]_{C^{2}(0,0)}(|y'|^2 + \sqrt{|s|})^2 \quad \text{for all } (x,t) \in \partial_p \widehat{\Omega} \cap Q_1.
	\end{align*}
	Finally, for $K\coloneqq \|\widehat{u}\|_{L^{\infty}(\widehat{\Omega} \cap Q_1)}+\|\widehat{f}\|_{L^{\infty}(\widehat{\Omega} \cap Q_1)}/\delta + 2^{2+\alpha}\|\widehat{g}\|_{C^{1,\alpha}(0,0)}/\delta+1$, by considering  
	\begin{align*}
		\widetilde{u}(x,t) \coloneqq \widehat{u}(x,t)/K, \quad \widetilde{f}(x,t) \coloneqq \widehat{f}(x,t)/K, \quad \text{and} \quad \widetilde{g}(x,t) \coloneqq \widehat{g}(x,t)/K,
	\end{align*}
and taking $\rho$ small enough (depending only on $n$, $\lambda$, $\alpha$, and $\|\partial_p \Omega \|_{C^{2}(0,0)}$), we may assume that 
\begin{equation*}
\|u\|_{L^{\infty}(\Omega \cap Q_1)} \leq 1, \quad  \|f\|_{L^{\infty}(\Omega \cap Q_1)} \leq \delta, \quad \|\partial_p\Omega\|_{C^{1,\alpha}(0,0)} \leq \frac{1-\eta^\alpha}{2C_1}\delta,
\end{equation*}
and 
	\begin{equation*} 
		|g(x,t)| \leq \frac{\delta}{2^{2+\alpha}} (|x|+\sqrt{|t|})^{1+\alpha} \quad \text{for all $(x,t) \in \partial_p\Omega  \cap Q_1$}.
	\end{equation*}
Thus, we can apply \Cref{thm:u-Lk_bd} for $u$ to find a sequence $\{a_k\}_{k=-1}^{\infty}$ satisfying \eqref{eq:lin1_bd} and \eqref{eq:lin2_bd}. Indeed, by standard argument, we obtain limit $a$ such that $a_k \to a$ satisfying
\begin{align*}
	|a_k-a| \leq C\eta^{k \alpha }.
\end{align*}
Finally, for any $(x,t) \in \Omega \cap Q_1$, there exists $j \geq 0$ such that $\eta^{j+1} \leq \max\{ |x|, \sqrt{|t|} \}< \eta^j$. Then we conclude that
\begin{align*}
	|u(x,t)-ax_n| \leq |u(x,t)-a_j x_n|+|a_j -a| |x_n| \leq C(|x|+\sqrt{|t|})^{1+\alpha},
\end{align*}
which implies that $u$ is $C^{1, \alpha}$ at $(0,0)$.
\end{proof}

\section{Parabolic normalized $p$-Laplace equations}\label{sec:normalized}
We begin with the global $C^{1, \alpha}$-estimate by combining the interior estimate \Cref{thm:main1} and the pointwise boundary estimate \Cref{thm:bdry}.
\begin{proof}[Proof of \Cref{thm:global}]
	For any $u \in S^{\ast}(\lambda,\Lambda, f)$ and linear function $L$, we know that $u-L$ is contained in $S^{\ast}(\lambda,\Lambda,f)$ again. A combination of \Cref{thm:main1} and \Cref{thm:bdry} allows us to conclude that solutions of \eqref{eq:diric} have global estimate; see {\cite[Theorem 3.1]{Wan92b}} for details.
\end{proof}
	
Let us move on to the regularity theory for parabolic normalized $p$-Laplace equations \eqref{eq:n_p_lap} for $p \in (1, \infty)$. We recall that
\begin{align*}
	\Delta_p^Nu=(\delta_{ij}+(p-2)|Du|^{-2}u_iu_j)u_{ij} 
\end{align*}
and
\begin{align*}
	\min\{p-1, 1\} I \leq \delta_{ij}+(p-2)q_iq_j|q|^{-2} \leq \max\{p-1, 1\}I \quad \text{for all $q \in \mathbb{R}^n \setminus \{0\}$.}
\end{align*}
Roughly speaking, we can expect that if $u$ is a viscosity solution of \eqref{eq:n_p_lap}, then $u \in S^{\ast}(\lambda_p, \Lambda_p, f)$ for $\lambda_p\coloneqq \min\{p-1, 1\}$ and $\Lambda_p \coloneqq \max\{p-1, 1\}$. However, a singularity arises when $Du=0$ and so we should be careful for the points where the gradient vanishes. To be precise, we are going to consider approximated solutions $u^{\varepsilon}$ with the following equations:
\begin{align}\label{eq:approx}
	\partial_tu^{\varepsilon}-a_{ij}^{\varepsilon}(Du^{\varepsilon}) D_{ij} u^{\varepsilon} =f \quad \text{in $Q_1$},
\end{align}
where the coefficient matrix $a_{ij}^{\varepsilon} : \mathbb{R}^n \to \mathcal{S}^n$ is defined by
\begin{align*}
	a_{ij}^{\varepsilon}(q)\coloneqq \delta_{ij}+(p-2)\frac{q_iq_j}{|q|^2+\varepsilon^2} \quad \text{for $q \in \mathbb{R}^n$}.
\end{align*}
Before we prove the global $C^{1, \alpha}$-regularity for solutions of \eqref{eq:n_p_lap} when $p$ is sufficiently close to $2$, we provide several preliminary results for the approximated equations \eqref{eq:approx}.

We are now ready to prove \Cref{thm:normalized}.
\begin{proof}[Proof of \Cref{thm:normalized}]
	Let $\omega$ be the modulus of continuity of $g$ in $\partial_p Q_1$. If we define $f^{\varepsilon}$ by the standard mollification of $f \in C(Q_1)$, then $f^{\varepsilon} \to f$ uniformly on compact subsets of $Q_1$. We approximate $u$ by regularizing the equation \eqref{eq:approx}. More precisely, let $v^{\varepsilon} \in C^{\infty}(Q_1)\cap C(\overline{Q_1})$ be the unique solution of the Dirichlet problem
		\begin{align*}
		\left\{
		\begin{aligned}
			\partial_tv^{\varepsilon}-a_{ij}^{\varepsilon}(Dv^{\varepsilon})D_{ij} v^{\varepsilon}&=f^{\varepsilon} && \text{in $Q_1$}\\
			v^{\varepsilon}&=g && \text{on $\partial_pQ_1$}.
		\end{aligned} \right.
	\end{align*}
The well-posedness of such problems is guaranteed by \cite{LSU68, Lie96}. Then it follows from the maximum principle that
\begin{align*}
	\|v^{\varepsilon}\|_{L^{\infty}(Q_1)} \leq \|g\|_{L^{\infty}(\partial_pQ_1)}+ C\|f^{\varepsilon}\|_{L^{\infty}(Q_1)}.
\end{align*}
Moreover, by following the proof of global H\"older estimates (\cite[Proposition 4.14]{CC95} or \cite[Proposition 2.5]{JS17}), we have
\begin{align*}
	|v^{\varepsilon}(x, t)-v^{\varepsilon}(y,s)| \leq \omega^{\ast}(|x-y|+\sqrt{|s-t|}) \quad \text{for all $(x, t), (y, s) \in \overline{Q_1}$}.
\end{align*}
Therefore, by Arzela-Ascoli theorem, we can extract a subsequence $\{v^{\varepsilon_k}\}$ such that $v^{\varepsilon_k} \to v \in C(\overline{Q_1})$ uniformly in $\overline{Q_1}$. We note that by the stability theorem and the comparison principle, we obtain that $u \equiv v$ in $\overline{Q_1}$.

On the other hand, we now concentrate on the uniform $C^{1, \alpha}$-estimate for $\{v^{\varepsilon}\}$. Indeed, the smoothness of $v^{\varepsilon}$ in $Q_1$ implies that $|Dv^{\varepsilon}(x, t)| < \infty$ for any $(x, t) \in Q_1$ and so 
\begin{align*}
	\left\{\begin{aligned}
	v^{\varepsilon} &\in S^{\ast}(\lambda_p, \Lambda_p, f^{\varepsilon}) && \text{in $Q_1$}\\
	v^{\varepsilon}&=g && \text{on $\partial_pQ_1$}
\end{aligned} \right.
\end{align*}
for $\lambda_p=\min\{p-1, 1\}$ and $\Lambda_p=\max\{p-1, 1\}$. Then an important observation is that 
\begin{align*}
	\Lambda_p/\lambda_p \to 1 \quad \text{when $p \to 2$}.
\end{align*}
Hence, in view of \Cref{thm:global}, there exists $\delta>0$ depending only on $n$ and $\alpha$ such that if $|p-2| <\delta$, then we have
	\begin{align*}
	\|v^{\varepsilon_k}\|_{C^{1, \alpha}(\overline{Q_{1}})} \leq C( \|u\|_{L^{\infty}(Q_1)} + \|f^{\varepsilon_k}\|_{L^{\infty}(Q_1)} +  \|g\|_{C^{1,\alpha}(\partial_pQ_1)}),
\end{align*}
where $C>0$ depends only on $n$, $p$, and $\alpha$. Finally, letting $k \to \infty$, we arrive at the desired conclusion.
\end{proof}

\bibliographystyle{abbrv}
\bibliography{references}

\end{document}